\documentclass[11pt,a4paper]{amsart}
\pdfoutput=1
\usepackage{hyperref,amsthm,amssymb,amsfonts,amsmath,float,graphicx}
\usepackage[numbers,sort&compress]{natbib}
\usepackage[lmargin=34mm,rmargin=34mm,tmargin=34mm,bmargin=34mm]{geometry}
\usepackage{lineno}
\usepackage{caption,subcaption}
\theoremstyle{plain}
\newtheorem{theorem}{Theorem}

\newtheorem{lemma}[theorem]{Lemma}

\newtheorem{observation}[theorem]{Observation}

\theoremstyle{remark} 
	\newtheorem{claim}{Claim}

\newcommand{\figlabel}[1]{\label{fig:#1}}
\newcommand{\figref}[1]{Figure~\ref{fig:#1}}

\newcommand{\spacing}[1]{\renewcommand{\baselinestretch}{#1}\setlength{\footnotesep}{\baselinestretch\footnotesep}}
\spacing{1.14}

\graphicspath{ {./} {newpics/}}
\newcommand{\Figure}[4][htb]{
\begin{figure}[#1]
	\vspace*{1ex}
	\begin{center}#3\end{center}
	\vspace*{-1ex}
	\caption{\figlabel{#2}#4}
\end{figure}
}

\DeclareMathOperator{\ES}{ES}
\DeclareMathOperator{\conv}{conv}

\newcommand{\bc}{\begin{center}}
\newcommand{\ec}{\end{center}}

\newcommand{\resultconst}{8}

\begin{document}
\sloppy

\title{Empty pentagons in point sets
 with collinearities}

\thanks{This research was supported by the DAAD and the Go8 within the Australia--Germany Joint Research Co-operation Scheme 2011/12 as part of the project \emph{Problems in geometric graph theory} (Kennz. 50753217).
}

\thanks{
J\'anos Bar\'at is supported by the Hungarian National Science
Foundation (OTKA) Grant K~76099,
and Australian Research Council (ARC) grant
DP120100197.
Vida Dujmovi\'c is supported by the Natural Sciences and Engineering
Research Council (NSERC) of Canada, and by an Endeavour Fellowship
from the Australian Government.
Gwena\"el Joret is a Postdoctoral Researcher of the Fonds
National de la Recherche Scientifique (F.R.S.--FNRS), and is also
supported by an Endeavour Fellowship. 
Michael Payne is supported by an Australian Postgraduate Award from the Australian Government.
Ludmila Scharf is supported by the German Research Foundation (DFG) Grant AL 253/7-1.
Daria Schymura was supported by the DFG within the Priority Programme 1307 \textit{Algorithm Engineering}.
Pavel Valtr is supported by the Ministry of Education of the Czech Republic under project CE-ITI (GA\v CR P2020/12/G061).
David Wood is supported by a QEII Research Fellowship from the ARC}


\date{\today}

\author[]{J\'anos Bar\'at}
\address{\newline School of Mathematical Sciences
\newline Monash University
\newline Victoria 3800, Australia}
\email{janos.barat@monash.edu}   

\author[]{Vida Dujmovi\'c}
\address{\newline School of Computer Science
\newline Carleton University
\newline Ottawa, Canada}
\email{vida@scs.carleton.ca}

\author[]{Gwena\"el Joret}
\address{\newline  D\'epartement d'Informatique
\newline Universit\'e Libre de Bruxelles
\newline Brussels, Belgium}
\email{gjoret@ulb.ac.be}

\author[]{Michael~S.~Payne}
\address{\newline Department of Mathematics and Statistics, 
\newline The University of Melbourne
\newline Melbourne, Australia}
\email{m.payne3@pgrad.unimelb.edu.au}

\author[]{Ludmila Scharf}
\address{\newline Institut f\"ur Informatik
\newline Freie Universit\"at Berlin
\newline Berlin, Germany}
\email{scharf@mi.fu-berlin.de}

\author[]{Daria Schymura}
\address{\newline Institut f\"ur Informatik
\newline Freie Universit\"at Berlin
\newline Berlin, Germany}
\email{daria.schymura@gmx.net}

\author[]{Pavel Valtr}
\address{\newline Department of Applied Mathematics and Institute for Theoretical Computer Science (CE-ITI), 
\newline Charles University
\newline Prague, Czech Republic}

\author[]{David~R.~Wood}
\address{\newline Department of Mathematics and Statistics, 
\newline The University of Melbourne
\newline Melbourne, Australia}
\email{woodd@unimelb.edu.au}

\begin{abstract}
An empty pentagon in a point set $P$ in the plane is a set of five points in $P$ in strictly convex position with no other point of $P$ in their convex hull.
We prove that every finite set of at least $328\ell^2$ points in the plane contains an empty pentagon or $\ell$ collinear points. This is optimal up to a constant factor since the $(\ell -1)\times(\ell-1)$ grid  contains no empty pentagon and no $\ell$ collinear points. The previous best known bound was doubly exponential.
\end{abstract}



\maketitle

\section{Introduction}\label{intro}

The Erd\H os-Szekeres Theorem~\cite{erdosszek}, a classical result in discrete geometry, states that for every integer $k$ there is a minimum integer $\ES(k)$ such that every set of at least $\ES(k)$ points in general position in the plane contains $k$ points in convex position. 
Erd\H os~\cite{erdos} asked whether a similar result held for empty $k$-gons ($k$ points in convex position with no other points inside their convex hull). 
Horton~\cite{horton} answered this question in the negative by showing that there are arbitrarily large point sets in general position that contain no empty heptagon.
On the other hand, Harborth~\cite{harborth} showed that every set of at least $10$ points in general position contains an empty pentagon.
More recently, Nicol\'as~\cite{nicolas} and Gerken~\cite{gerken} independently settled the question for $k=6$ by showing that sufficiently large point sets in general position always contain empty hexagons; see also~\cite{koshelev,pavel}.

These questions are not interesting if the general position condition is abandoned completely, since a collinear point set contains no three points in convex position.
However, considering point sets with a bounded number of collinear points does lead to interesting generalisations of these problems.
First some definitions are needed.
A point set $X$ in the plane is in \emph{weakly convex position} if every point in $X$ lies on the boundary of $\conv(X)$, the convex hull of $X$.
A point $x \in X$ is a \emph{corner} of $X$ if $\conv(X \setminus \{x\}) \neq \conv(X)$.
The set $X$ is in \emph{strictly convex position} if every point in $X$ is a corner of $X$.
A \emph{weakly} (respectively \emph{strictly}) \emph{convex $k$-gon} is a set of $k$ points in weakly (respectively strictly) convex position.
It is well known that the Erd\H os-Szekeres theorem generalises for point sets with bounded collinearities; see~\cite{abeletal} for proofs. 
One generalisation states that every set of at least $\ES(k)$ points contains a weakly convex $k$-gon. 
For strictly convex position, the generalisation states that for all integers $k$ and $\ell$ there exists a minimum integer $\ES(k,\ell)$ such that every set of at least $\ES(k,\ell)$ points in the plane contains $\ell$ collinear points or a strictly convex $k$-gon.

This paper addresses the case of empty pentagons in point sets with collinearities.
A subset $X$ of a point set $P$ is an \emph{empty $k$-gon} if $X$ is a strictly convex $k$-gon 
and $P \cap \conv(X) = X$.
Abel et al.~\cite{abeletal} showed that every finite set of at least $\ES(\frac{(2\ell-1)^\ell-1}{2\ell-2})$ points in the plane contains an empty pentagon or $\ell$ collinear points. 
The function $\ES(k)$ is known to grow exponentially~\cite{erdosszek,erdosszek2}, so 
this bound is doubly exponential in $\ell$.
See \citep{Rabinowitz,Eppstein-JoCG} for more on point sets with no empty pentagon.
In the present paper the following theorem is proved without applying the Erd\H os-Szekeres Theorem.
\begin{theorem}\label{mainthm}
Let $P$ be a finite set of points in the plane.
If $P$ contains at least $328\ell^2$ points, then $P$ contains an empty pentagon or $\ell$ collinear points.
\end{theorem}
This quadratic bound is optimal up to a constant factor since the $(\ell-1) \times (\ell-1)$ square grid has $(\ell-1)^2$ points and contains neither an empty pentagon nor $\ell$ collinear points.

Concerning the general question of the existence of empty $k$-gons in point sets with collinearities, Horton's negative result for empty heptagons also applies in this setting. However, it is not clear how to adapt the proofs of Nicol\'as and Gerken to deal with collinearities, and the case $k=6$ remains open.

The point set $P$ will be assumed to be finite throughout this paper, and indeed Theorem~\ref{mainthm} does not hold for infinite sets. A countably infinite point set in general position with no empty pentagons can be constructed recursively from any finite set in general position by repeatedly placing points inside every empty pentagon, avoiding collinearities. On the other hand, Theorem~\ref{mainthm} easily generalises to \emph{locally finite} point sets, point sets which contain only finitely many points in any bounded region. The result of Abel et al.~\cite{abeletal} already implies that an infinite locally finite set with no empty pentagon contains $\ell$ collinear points for every positive integer $\ell$.

The remainder of this section introduces terminology that is used throughout the paper.
The \emph{convex layers} $L_1, \dots, L_r$ of $P$ are defined recursively as follows: $L_i$ is the subset of $P$ lying in the boundary of the convex hull of $P \setminus \bigcup_{j=1}^{i-1} L_j$,
and $L_r$ is the innermost layer, so $P= \bigcup_{i=1}^r L_i$ and $L_i \neq \emptyset$ for $i=1,\dots, r$.
Note that each layer is in weakly convex position.

Points of $P$ will also be referred to as \emph{vertices} and line segments connecting two points of $P$ as \emph{edges}.
The \emph{edges of a layer} are the edges between consecutive points in the boundary of the convex hull of that layer.
Edges of layers will always be specified in clockwise order.
A single letter such as $e$ is often used to denote an edge.
For an edge $e$, let $l(e)$ denote the line containing $e$. 
Some edges will be used to determine half-planes.
The open half-planes determined by $l(e)$ will be denoted $e^+$ and $e^-$, where the $+$ and $-$ sides will be determined later.
Similarly, the closed half-planes determined by $l(e)$ will be denoted $e^\oplus$ and $e^\ominus$.

Gerken~\cite{gerken} introduced the notion of $k$-sectors.
If $p_1p_2p_3p_4$ is a strictly convex quadrilateral (that is, a strictly convex $4$-gon), then the \emph{$4$-sector} $S(p_1,p_2,p_3,p_4)$ is the set of all points $q$ such that $qp_1p_2p_3p_4$ is a strictly convex pentagon. 
Note that the order of the arguments is significant.
$S(p_1,p_2,p_3,p_4)$ is the intersection of three open half-planes, and may be bounded or unbounded, as shown in Figure~\ref{sector}.
The closure of a $4$-sector will be denoted by square brackets, $S[p_1,p_2,p_3,p_4]$.
If $P$ contains no empty pentagon and $p_1p_2p_3p_4$ is an empty quadrilateral in $P$, then $P \cap S(p_1,p_2,p_3,p_4) = \emptyset$. 
Otherwise, since $P$ is finite, there exists a point $x \in P \cap S(p_1,p_2,p_3,p_4)$ closest to the line $l(p_1p_4)$, and $xp_1p_2p_3p_4$ is an empty pentagon.

\begin{figure}
\bc \includegraphics
{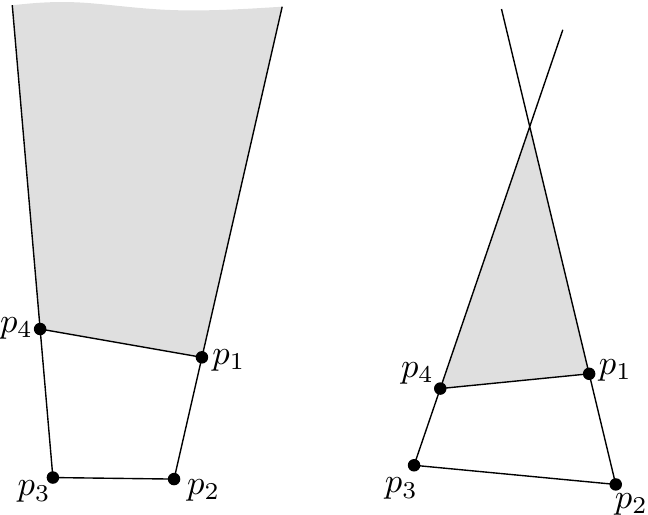} \ec
\caption{The shaded regions represent the $4$-sector $S(p_1,p_2,p_3,p_4)$, which may be bounded or unbounded.}
\label{sector}
\end{figure}


\section{Large subsets in weakly convex position}

The first major step in proving Theorem~\ref{mainthm} is to establish the following theorem concerning point sets with large subsets in weakly convex position. 

\begin{theorem}\label{convposthm}
If a point set $P$ contains $8\ell$ points in weakly convex position, then $P$ contains an empty pentagon or $\ell$ collinear points.
\end{theorem}

A similar result has been obtained independently by Cibulka and Kyn\v{c}l~\cite{privcom}. Theorem~\ref{convposthm} immediately implies that every point set with $\ES(8\ell)$ points contains an empty pentagon or $\ell$ collinear points, which is already a substantial improvement on the result of Abel et al.~\cite{abeletal} mentioned above. The rest of this section is dedicated to proving it.

Throughout this section, let $P$ be a set of points in the plane that contains $\resultconst\ell$ points in weakly convex position but contains no $\ell$ collinear points. 
Suppose for the sake of contradiction that $P$ contains no empty pentagon.
Let $A$ be an inclusion-minimal weakly convex $8\ell$-gon in $P$. 
That is, there is no weakly convex $8\ell$-gon $A'$ such that $\conv(A') \subsetneq \conv(A)$.
An empty pentagon in $P \cap \conv(A)$ is an empty pentagon in $P$, so it can be assumed that $P \subseteq \conv(A)$, so $A$ is the first convex layer of $P$. 
Let $B$ be the second convex layer of $P$. 
For an edge $e$ of $A$ or $B$, let $e^+$ be the open half-plane determined by $l(e)$ that does not contain any point in $B$. 

\begin{observation}\label{moreabove}
For each edge $b$ of $B$, $| A \cap b^+ | > |B \cap l(b)|$. 
Similarly, if $b_1,b_2,\dots, b_j$ are edges of $B$, then $$\left|A \cap \bigcup_{i=1}^j b_i^+ \right| > \left| B \cap \bigcup_{i=1}^j l(b_i) \right|.$$
\end{observation}

\begin{proof}
If $|A \cap b^+| \leq |B \cap l(b)|$ then removing the vertices $A \cap b^+$ from $A$ and replacing them by $B \cap l(b)$ gives a weakly convex $m$-gon $Q$ such that $m\geq |A|$ and $\conv(Q) \subsetneq \conv(A)$, contradicting the minimality of $A$; see Figure~\ref{fig:obsa}. 
The second claim follows from the minimality of $A$ in a similar way. 
\end{proof}

\begin{figure}
\begin{subfigure}[t]{0.4\textwidth}
\centering
\includegraphics{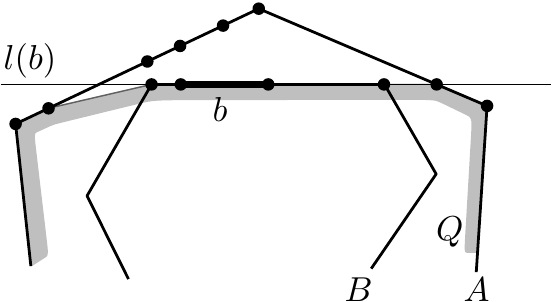} 
\caption{}
\label{fig:obsa}
\end{subfigure}
\hspace{1cm} 
\begin{subfigure}[t]{0.4\textwidth}
\centering
\includegraphics{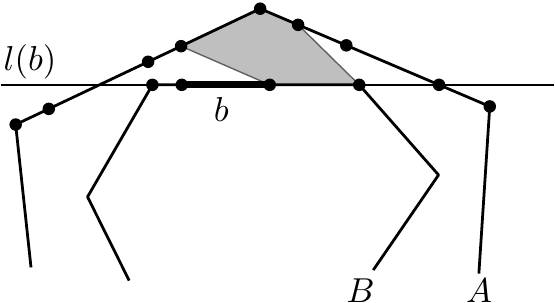} 
\caption{}
\label{fig:obsb}
\end{subfigure}
\caption{(a) If $|A \cap b^+| \leq |B \cap l(b)|$, then $A$ is not minimal. (b) If $b^+$ contained three non-collinear points of $A$, there would be an empty pentagon.}
\end{figure}

\begin{observation}\label{collinear}
For each edge $b$ of $B$, the vertices of $A \cap b^+$ are collinear. 
\end{observation}

\begin{proof}
By Observation~\ref{moreabove}, there are at least 3 points in $A \cap b^+$. 
If $A \cap b^+$ is not collinear, then there is an empty pentagon; see Figure~\ref{fig:obsb}. 
\end{proof}

The following lemma implies that $B$ has at least $4\ell$ vertices.

\begin{lemma}\label{double} $2|B| \geq |A|$.
\end{lemma}
\begin{proof}
%
Since $|A| \geq 8\ell$, $A$ has at least nine corners. 
Thus $B\neq \emptyset$. 
%
If $B$ is collinear 
then let $h$ be the line containing $B$. 
There are at most two corners of $A$ on $h$, so there are at least four corners of $A$ strictly to one side of $h$.
The interior of the convex hull of these four corners together with any point in $B$ is empty.
This implies that there is an empty pentagon in $P$, a contradiction.

Therefore $B$ has at least three corners, and at least three sides, where a \emph{side} of $B$ is the set of edges between consecutive corners.
Let $b_1, \dots, b_k$ be edges of $B$, one in each side of $B$.
By Observation~\ref{collinear}, each of the sets $A \cap b_i^+$ is collinear for $i = 1, \dots, k$.
Thus $|A|\leq \Sigma_{i=1}^k |A \cap b_i^+| < k\ell$, and so $k\geq 9$.
In other words, $B$ has at least nine corners,
so there is at least one point $z \in P$ in the interior of $\conv(B)$.
Suppose that for some edge $xy$ of $A$ the closed triangle $\Delta[x,y,z]$ contains no point of $B$.
Then there is an edge $x'y'$ of $B$ that crosses this triangle.
The $4$-sector $S(x',x,y,y')$ contains $z$, contradicting the fact that $P$ contains no empty pentagon.
Thus every such closed triangle contains a point of $B$.
Since each point of $B$ is in at most two such closed triangles, $2|B| \geq |A|$.
%
\end{proof}

The following lemma implies that for a set of points $X$, the first edge $b$ in $B$ in clockwise order such that $X \subseteq b^+$ is well defined, as long as there is at least one such edge.

\begin{lemma}\label{paths}
For any set of points $X\neq \emptyset$, let $E_X$ be the set of edges $b$ in $B$ such that $X\subseteq b^+$. Then the edges in $E_X$ are consecutive in $B$, and not every edge of $B$ is in $E_X$.
\end{lemma}
\begin{proof}
If $X \cap \conv(B)\neq \emptyset$ then $E_X = \emptyset$. 
Take a point $x\in X$, so $x \not\in \conv(B)$.
Let $y$ be a point in the interior of $\conv(B)$ that is not collinear with any two points of $B \cup\{x\}$.
Then $l(xy)$ intersects precisely two edges $b$ and $\tilde{b}$ of $B$, with $x\in b^+$ and $x \in \tilde{b}^-$. 
Thus, $X\not\subseteq \tilde{b}^+$, so $E_X$ does not contain every edge of $B$.

If $E_X$ contains only one edge then the lemma holds, so consider two edges $b_1$ and $b_2$ in $E_X$ and suppose they are not consecutive.
If $l(b_1)=l(b_2)$, then clearly the edges between $b_1$ and $b_2$ on $l(b_1)$ are also in $E_X$.
Now suppose $l(b_1)\neq l(b_2)$.
If $l(b_1)$ and $l(b_2)$ are parallel, then $b_1^+ \cap b_2^+ = \emptyset$, a contradiction. 
So $l(b_1)$ and $l(b_2)$ cross at a point $p$. 
Without loss of generality, $p$ is above $B$ with $b_1$ on the left and $b_2$ on the right, as shown in Figure~\ref{fig:pathsa}.
%
Let $b$ be the next edge clockwise from $b_1$. 
Then clearly $p \in b^\oplus$, so $b_1^+ \cap b_2^+ \subseteq b^+$, and hence $b \in E_X$.
Iterating this argument shows that every edge clockwise from $b_1$ until $b_2$ is in $E_X$.
It follows that the edges in $E_X$ are consecutive in $B$.
\end{proof}

\begin{figure}
\centering
\includegraphics{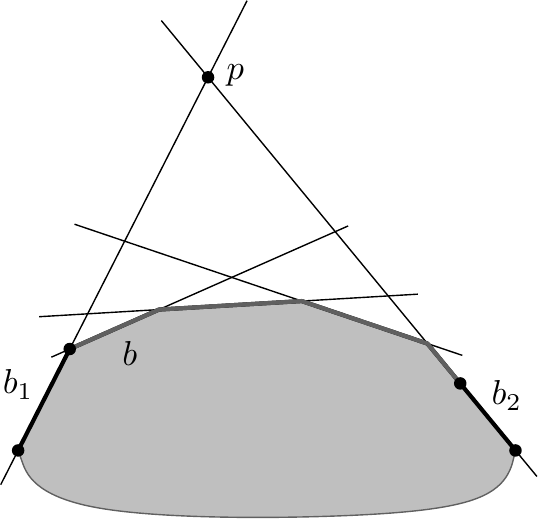}
\caption{Lemma~\ref{paths}.}
\label{fig:pathsa}
\end{figure}

Let $a$ be an edge of $A$ such that $|A \cap l(a)| \geq 3$. 
Such an edge exists by Observations~\ref{moreabove} and \ref{collinear}. 
Let $\{v_1,\dots,v_k\}$ be $A \cap l(a)$ in clockwise order. Thus $k < \ell$.

\begin{lemma}\label{3ptlemma}
There is an edge $b$ of $B$ such that $\{v_1,v_2,v_3\}\subseteq b^+$ or $\{ v_{k-2},v_{k-1},v_k \}\subseteq b^+$.
\end{lemma}

\begin{proof}
Let $b$ be an edge of $B$ with $v_2\in b^+$. Such an edge exists, since otherwise $v_2\in\conv(B)$.
Observations~\ref{moreabove} and \ref{collinear} imply that $|A \cap b^+|\geq 3$ and $A \cap b^+$ is collinear. 
Thus if $v_1 \in b^+$, then $\{v_1,v_2,v_3\} \subseteq b^+$, as required. 
Otherwise $l(b)$ intersects $l(a)$ between $v_1$ and $v_2$, 
so $\{v_2,v_3,\dots,v_k\} \subseteq b^+$ and $k\geq 4$, because if $k=3$ then $|A \cap b^+| =2$. 
\end{proof}

By Lemma~\ref{3ptlemma}, without loss of generality, there is an edge $b$ of $B$ such that 
$\{v_1,v_2,v_3\} \subseteq b^+$, and by Lemma~\ref{paths} the edges with this property are consecutive in $B$. 
Let $b_1$ be the first one in clockwise order. 
For an illustration of the following definitions, see Figure~\ref{fig:sectorsAndRays}. 
First observe that $|A \cap l(a) \cap \tilde{b}^+| \geq 3$ cannot
hold for every edge $\tilde{b}$ of $B$, because otherwise $A \cap l(a)
= A$ by Observation~\ref{collinear}, and so $|A|<\ell$.
Define the endpoints of $b_1$ to be $w_1$ and $w_2$ in clockwise order.
Let $w_3,\dots,w_{m+1}$ and $b_i:=w_iw_{i+1}$ be subsequent vertices and edges of $B$ in clockwise order, 
where $|A \cap l(a) \cap b_{m-1}^+|\geq 3$ but $|A \cap l(a) \cap b_{m}^+|\leq 1$. 
Then $m \leq |B \cap \bigcup_{i=1}^{m-1} l(b_i)|<|A  \cap \bigcup_{i=1}^{m-1} b_{i}^+|\leq k$ by Observation~\ref{moreabove}. 
%
Now define $e_i:=v_iw_i$ for $i=1,\dots,m$. 
Let $e_i^-$ be the open half-plane determined by $l(e_i)$ that contains $v_1$, or that does not contain $v_2$ in the case of $e_1$.

\begin{figure}
\centering
\includegraphics
{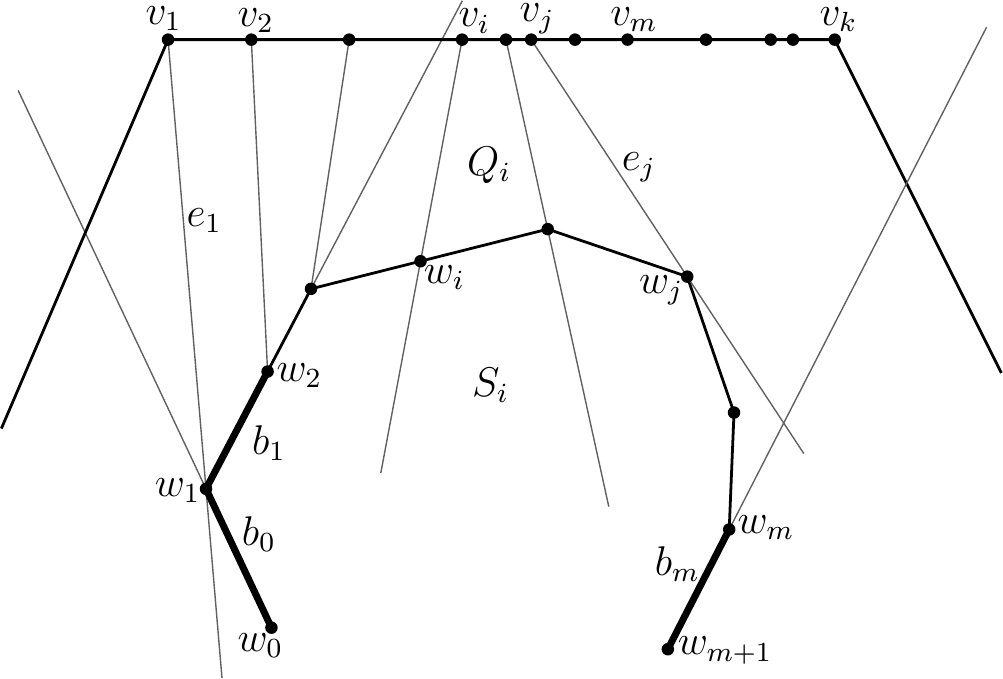} 
\caption{Definition of $b_1$ and the quadrilaterals $Q_i$.}
\label{fig:sectorsAndRays}
\end{figure}

Let $j$ be minimal such that the closed half-plane $e_j^\ominus$ contains $B$. 
Clearly $j\neq1$ since $w_2 \in e_1^+$.
The following argument shows that $j$ is well-defined.
%
%
Call $e_i$ \emph{good} if $w_i$ is the closest point of $l(e_i) \cap \conv(B)$ to $v_i$. 
%
%
First suppose that $e_m$ is good, so in particular $v_m \in b_{m-1}^+$. 
Since $m$ was chosen so that $|A \cap l(a) \cap b_{m-1}^+|\geq 3$ but $|A \cap l(a) \cap b_{m}^+|\leq 1$, and since $m<k$, it follows that $v_m\in b_{m}^\ominus$ also.
This implies that $B\subseteq e_m^\ominus$, as illustrated in Figure~\ref{fig:intermediateVertexa}, and so $j$ is well-defined. 
Now suppose that $e_m$ is not good.
By the choice of $b_1$, both $e_1$ and $e_2$ are good, so let $p$ be minimal such that $e_p$ is not good. Thus $3 \leq p \leq m$. 
Then $w_{p-2}$ is in $e_{p-1}^-$ because $e_{p-1}$ is good, and $w_p$ is in $e_{p-1}^-$ because $e_p$ is not good, as shown in Figure~\ref{fig:intermediateVertexb}.
This implies that $B \subseteq e_{p-1}^\ominus$, 
so $j$ is well-defined. 
Note that this also shows that $e_i$ is good for all $i=1,\dots,j$.



\begin{figure}
\begin{subfigure}[t]{0.4\textwidth}
\centering
\includegraphics{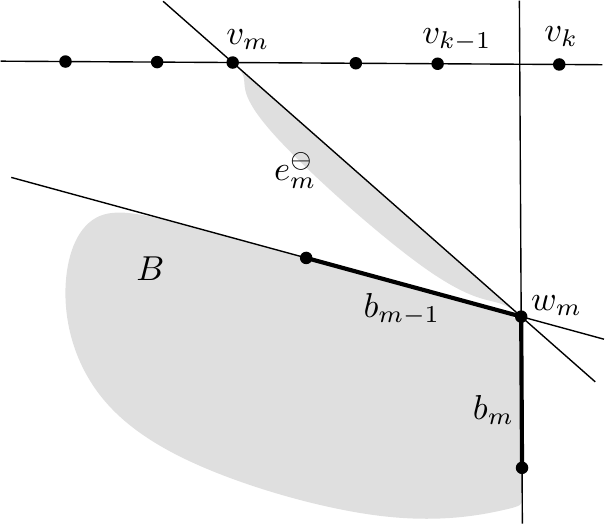} 
\caption{}
\label{fig:intermediateVertexa}
\end{subfigure}
\hspace{1cm} 
\begin{subfigure}[t]{0.4\textwidth}
\centering
\includegraphics{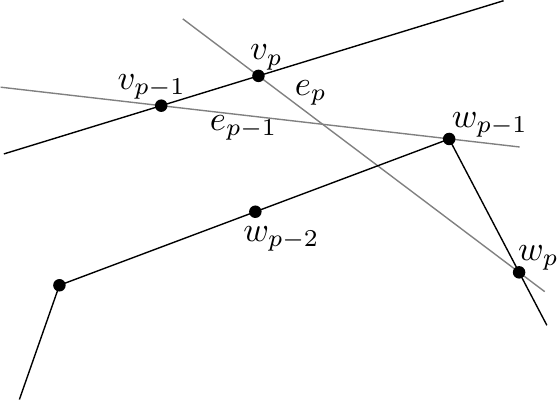}
\caption{}
\label{fig:intermediateVertexb}
\end{subfigure}
\caption{(a) If $e_m$ is good then $B \subseteq e_m^\ominus$. (b) If $e_{p-1}$ is good and $e_p$ is not, then $B\subseteq e_{p-1}^\ominus$.}
\end{figure}

Define the quadrilaterals $Q_i:=w_i v_i v_{i+1} w_{i+1}$ for $i=1,\dots,j-1$.
By the following argument, the quadrilaterals $Q_i$ are strictly convex. 
Suppose on the contrary that $Q_h$ is not strictly convex, and $h$ is minimal. 
There are two possible order types for $Q_h$.
The first possibility is that $v_h \in b_h^\ominus$ and so $B \subseteq e_h^\ominus$ (since $e_h$ is good), contradicting the minimality of $j$; 
see Figure \ref{fig:nonConvexQuadsa}.
The second possibility is that $v_{h+1} \in b_h^\ominus$ and so $A \cap \bigcup_{i=1}^{h} b_i^+ = \{v_1, \dots, v_h \}$, which contradicts Observation~\ref{moreabove} since $|B \cap \bigcup_{i=1}^h l(b_i)| \geq h+1 $; see Figure~\ref{fig:nonConvexQuadsb}.



\begin{figure}
\begin{subfigure}[t]{0.4\textwidth}
\centering
\includegraphics{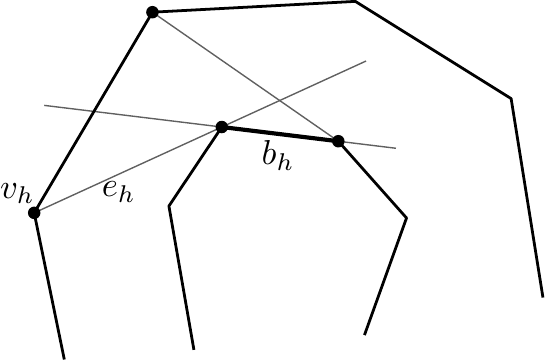} 
\caption{}
\label{fig:nonConvexQuadsa}
\end{subfigure}
\hspace{1cm} 
\begin{subfigure}[t]{0.4\textwidth}
\centering
\includegraphics{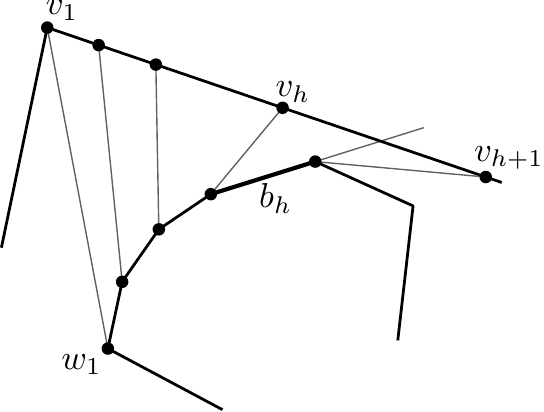}
\caption{}
\label{fig:nonConvexQuadsb}
\end{subfigure}
\caption{(a) If $v_h \in b_h^\ominus$ then $B \subseteq e_h^\ominus$. (b) If $v_{h+1} \in b_h^\ominus$ then $A \cap \bigcup_{i=1}^{h} b_i^+ = \{v_1, \dots, v_h \}$.}
\end{figure}

Let $S_i:=S[w_i,v_i,v_{i+1},w_{i+1}]$ be the closed 4-sector of the quadrilateral $Q_i$ for $i=1,\dots,j-1$.
Note that $B \cap S_i = B \cap e_i^\oplus \cap e_{i+1}^\ominus$. 
Take a point $x \in B \cap e_1^\oplus$.
Then  $x \in e_j^\ominus$ since $B \subseteq e_j^\ominus$.
Let $h$ be minimal such that $x \in e_{h+1}^\ominus$.
If $h = 0$ then $x \in l(e_1)\cap B \subseteq S_1$.
Otherwise $x \not\in e_h^\ominus$, so $x \in e_h^\oplus$, and so $x \in S_h$.
Hence $B \cap e_1^\oplus \subseteq \bigcup_{i=1}^{j-1} S_i$.

The quadrilaterals $Q_i$ are empty because they lie between the layers $A$ and $B$. 
%
Therefore no $S_i$ contains a point of $B$ in its interior, 
and so all the points of $B \cap e_1^\oplus$ lie on the lines $l(e_1), \dots, l(e_j)$. 
Since $B$ is in weakly convex position, $|B \cap l(e_i)| \leq 2$ for $i=2, \dots, j-1$.
There can be at most $\ell -2$ points of $B$ on $l(e_1)$ and $l(e_j)$. 
In fact there are less points of $B$ on $l(e_1)$ and $l(e_j)$, as the following argument shows.
Note that $w_m$ is a corner of $B$ since $A \cap b_{m-1}^+ \neq A \cap b_{m}^+$.
Therefore $B \cap l(e_j) \subseteq \{w_j,\dots,w_m\}$, and so $|B \cap l(e_j)| \leq m-j+1$. 
Since $j, m < \ell$, adding up the bounds for each $l(e_i)$ yields $|B \cap e_1^\oplus| \leq (\ell -2) + 2(j-2) + (m-j+1) < 3\ell.$ 
Since $|B| \geq 4\ell$ by Lemma~\ref{double}, this implies that $B \not\subseteq e_1^\oplus$, which implies that $|B \cap l(e_1)| \leq 2$.
Hence $|B \cap e_1^\oplus| \leq 2(j-1) + (m-j+1) < 2\ell$.

It remains to bound the size of the rest of $B$, that is, $|B \cap e_1^-|$.
Define $v_0, v_{-1}, v_{-2}, \dots$ and $w_0, w_{-1}, w_{-2}, \dots$ to be the vertices of $A$ and $B$ proceeding anticlockwise from $v_1$ and $w_1$ respectively. 
Define $b_0:=w_0w_1$.
Since $B \not\subseteq e_1^\oplus$, it follows that $v_1 \in b_0^+$, as shown in Figure~\ref{fig:coveringB}. 
%
%
Since $b_1$ is the first edge in clockwise order with $\{v_1, v_2,v_3\} \subseteq b_1^+$, neither $v_2$ nor $v_3$ is in $b_0^+$.
Hence by Observation~\ref{moreabove}, $\{v_1,v_0,v_{-1}\}\subseteq b_0^+$.
Also, by Observation~\ref{collinear}, neither $v_0$ nor $v_{-1}$ is in $b_1^+$, so $b_0$ is the first edge of $B$ with $\{v_1,v_0,v_{-1}\}\subseteq b_0^+$ in anticlockwise order 
(recall that edges with this property are consecutive in $B$ by Lemma~\ref{3ptlemma}).
Therefore, the argument that started at $b_1$ and proceeded clockwise may be started at $b_0$ and proceed anticlockwise instead.
In this situation, the edge $e_1$ will remain the same as before because the starting points $v_1$ and $w_1$ are unchanged.
Thus the argument will cover $B \cap e_1^\ominus$ with 4-sectors and, analogously to before, show that $|B \cap e_1^\ominus| < 2\ell$.
This implies that $|B| \leq |B \cap e_1^\ominus| + |B \cap e_1^\oplus| < 4\ell$, which contradicts the fact that $|B| \geq 4\ell$.
This completes the proof of Theorem~\ref{convposthm}. \qed

\begin{figure}
\bc \includegraphics
{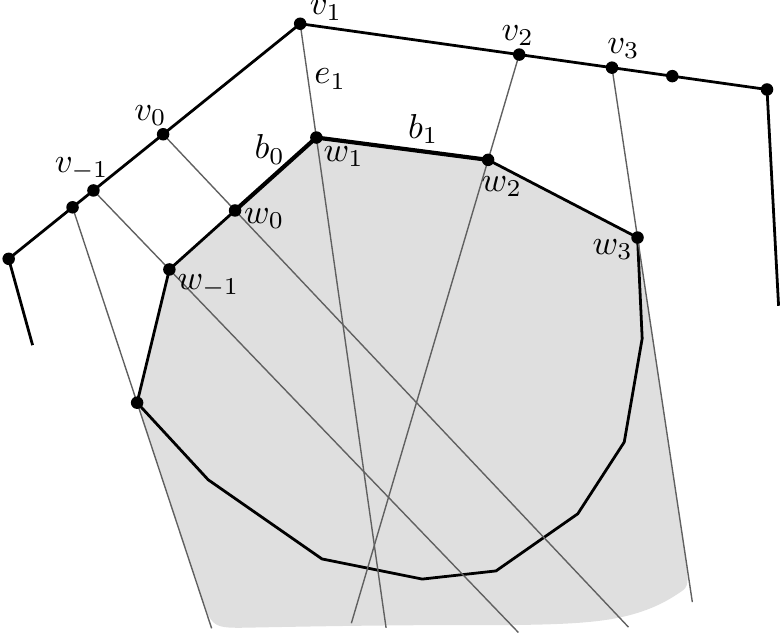} \ec
\caption{The convex hull of $B$ is covered by the union of the closed sectors $S_i$.}
\label{fig:coveringB}
\end{figure}


\section{Proof of Theorem~\ref{mainthm}}

Let $P$ be a set of at least $328\ell^2$ points with no $\ell$ collinear points, and suppose for the sake of contradiction that $P$ does not contain an empty pentagon.
Let $L_1, \dots, L_r$ be the convex layers of $P$, with $L_1$ the outermost and $L_r$ the innermost layer.
Theorem~\ref{convposthm} implies that $|L_i| < 8\ell$ for every $i$.
The layers are divided into three groups as follows.
The layers $L_{r-\ell +1}$ to $L_r$ are the \emph{inner layers}. Hence $\left| L_{r-\ell+1} \cup \dots \cup L_r \right| < 8\ell^2$.
The layers $L_1$ to $L_a$ are the \emph{outer layers}, 
where $a$ is the minimum integer such that $\left| L_1 \cup \dots \cup L_a \right| \geq 64\ell(\ell-1)$.
This means that $\left| L_1 \cup \dots \cup L_a \right| \leq 64\ell(\ell-1) + 8\ell < 64\ell^2$.
The remaining layers $L_{a+1}$ to $L_{r-\ell}$ are the \emph{middle layers}.

The strategy of the proof is to analyse the structure of the middle layers and show that if there are too many middle layers, then the outer layers contain less points than the lower bound in the previous paragraph.
This contradiction implies that there are not too many middle layers.
Since the size of each layer is limited by Theorem~\ref{convposthm}, this yields an upper bound on the number of points in the middle layers.
Adding this upper bound to those just established for the inner and outer layers will give a contradiction to the assumed size of $P$, completing the proof.

Abel et al.~\cite{abeletal} introduced the following definition and lemma.
Fix a point $z \in L_r$.
An edge $xy$ of $L_i$ is \emph{empty} if the open triangle $\Delta(x,y,z)$ contains no points of $L_{i+1}$. 
%
\begin{lemma}\label{empty}\emph{\cite{abeletal}}
If $L_i$ contains an empty edge for some $i \in \{1, \dots, r -\ell +1 \}$, then $P$ contains an empty pentagon or $\ell$ collinear points.
\end{lemma}
%
Lemma~\ref{empty} is not stated in this form in the paper by Abel et al.~\cite{abeletal}, so the proof 
 is included in Appendix~\ref{append} for completeness.

For now, consider only the points in the middle layers $L_{a+1}$ to $L_{r-\ell}$. 
For each point $v$ in a middle layer $L_i$, define the left and right child of $v$ as follows (see Figure~\ref{fig:childrena}).
Let $x$ be the closest point to $v$ in $\conv(L_{i+1}) \cap vz$ (where $vz$ is the line segment from $v$ to $z$).
The \emph{right child} of $v$ is the point in $L_{i+1}$ immediately clockwise from $x$.
The \emph{left child} of $v$ is the point in $L_{i+1}$ immediately anticlockwise from $x$.
Note that although $x$ may be in $P$, $x$ is neither the left nor the right child of $v$.

A \emph{right chain} is a sequence $v_1, \dots, v_t$ of points in $L_{a+1} \cup \dots \cup L_{r-\ell}$ such that $v_{i+1}$ is the right child of $v_i$.
A \emph{left chain} is defined in a similar fashion.
A \emph{subchain} is a chain contained in a larger chain, and a \emph{maximal} chain is one that is not a proper subchain of another chain.
A point cannot be the right child of two points $u$ and $v$ in $L_i$, otherwise the edge $uv$ (or the edges in the segment $uv$ if $u$ and $v$ are not adjacent) would be empty, contradicting Lemma~\ref{empty}. 
Similarly, a point cannot be the left child of two points.
This implies that maximal right chains do not intersect one another, and similarly for maximal left chains.
Furthermore, by construction each point in the middle layers has a left and a right child, so every maximal chain contains a point in $L_{r-\ell}$. 
Together these observations imply the following lemma.

\begin{lemma}\label{chains} 
Every point in the middle layers 
is in precisely one maximal right chain and one maximal left chain.
The number of maximal right chains is $|L_{r-\ell}| \leq 8\ell -1$, and similarly for maximal left chains. \qed
\end{lemma}


\begin{figure}
\begin{subfigure}[t]{0.4\textwidth}
\centering
\includegraphics{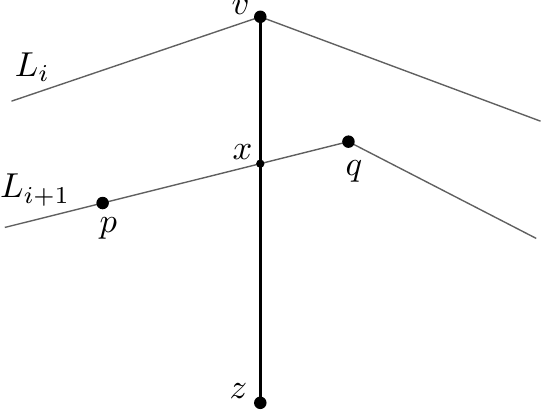} 
\caption{}
\label{fig:childrena}
\end{subfigure}
\hspace{1cm} 
\begin{subfigure}[t]{0.4\textwidth}
\centering
\includegraphics{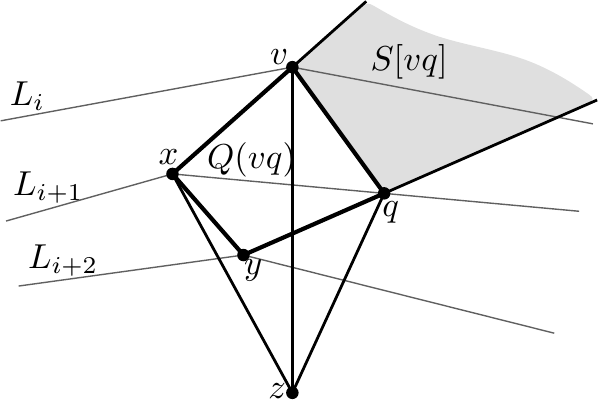}
\caption{}
\label{fig:childrenb}
\end{subfigure}
\caption{(a) The right child $q$ and the left child $p$ of $v$. (b) The quadrilateral $Q(vq)$ and the sector $S[vq]$.}
\end{figure}





The edges of a chain are the edges between consecutive vertices of the chain.
A chain $V$ is said to \emph{wrap around} if every ray starting at $z$ intersects the union of the edges of $V$ at least twice.
Since chains advance in the same direction around $z$ with every step, this is equivalent to saying that $V$ covers a total angle of at least $4\pi$ around $z$.

\begin{lemma}\label{wrap} 
If the number of middle layers $r-\ell-a$ is at least $32\ell$, then there is a chain with at most $32\ell$ vertices that wraps around.
\end{lemma}
\begin{proof}
Let $V=(v_1, \dots, v_t)$ be a right chain that starts at a point $v_1\in L_{a+1}$. 
Since $r-\ell-a \geq 32\ell$, it can be assumed that $t=32\ell$.
By Lemma~\ref{chains}, each vertex $v_i$ lies in some left chain, and there are at most $8\ell-1$ maximal left chains, so some left chain intersects $V$ at least five times.
Let $U$ be a left chain that intersects $V$ in the points $p_1, \dots, p_5$, where $p_1$ and $p_5$ are the first and last points of $U$ respectively.

Recall that right chains advance clockwise around $z$ with every step, and left chains anticlockwise.
Therefore, the paths from $p_i$ to $p_{i+1}$ in $U$ and $V$ form a closed curve around $z$.
So these paths cover an angle of $2\pi$ around $z$.
Hence $U$ and $V$ together cover a total angle of at least $8\pi$ around $z$.
This implies that at least one of them covers a total angle of at least $4\pi$, and thus wraps around.
Both $U$ and $V$ have at most $t$ vertices because they lie in the layers $L_{a+1}$ to $L_{a+t}$.
\end{proof}

If $q$ is the right child of a vertex $v$ in a middle layer $L_i$, then associate with $vq$ the following quadrilateral, as illustrated in Figure~\ref{fig:childrenb}.
Let $x$ be the point in $L_{i+1}$ anticlockwise from $q$, 
so $x$ either lies on $vz$ or is the left child of $v$.
Let $y$ be a point in the open triangle $\Delta(x,q,z)$ closest to $xq$. 
Such a $y$ exists in $L_{i+2}$, otherwise $xq$ would be an empty edge.
Then $Q(vq):=vxyq$ is the quadrilateral associated with $vq$.
This quadrilateral is strictly convex by construction.
The triangle $\Delta[x,q,y]$ is empty since $x$ and $q$ are neighbours in $L_{i+1}$ and $y$ is a closest point to $xq$.
The triangle $\Delta[v,q,x]$ is empty because it can contain neither a point of $L_i$ nor $L_{i+1}$. 
Thus $Q(vq)$ is an empty quadrilateral.
Empty quadrilaterals determine $4$-sectors that must be empty since there are no empty pentagons.
Let $S[vq]$ be the closed $4$-sector determined by $Q(vq)$, that is, $S[v,x,y,q]$ in the notation established previously.


%

Let $V=(v_1,\dots,v_t)$ be a chain and let $e_i := v_iv_{i+1}$ be the edges of $V$.
Let $e_i^\oplus$ be the closed half-plane defined by $e_i$ that does not contain $z$.
Consider a quadrilateral $Q(e_i)=v_ix_iy_iv_{i+1}$ and let $c_i$ be the edge $x_iv_i$ and let $d_i$ be the opposite edge $y_iv_{i+1}$.
Let $c_i^\oplus$ be $c_i^\oplus$ the closed half-plane defined by $c_i$ that contains $d_i$, and let $d_i^\oplus$ be the closed half-plane defined by $d_i$ that contains $c_i$.
%
With these definitions, the $4$-sector defined by $Q(e_i)$ is $S[e_i] = c_i^\oplus \cap d_i^\oplus \cap e_i^\oplus$.

\begin{lemma}\label{cover} 
If $V=(v_1, \dots, v_t)$ wraps around, then the 
corresponding $4$-sectors $S[e_i]$ cover the points of the outer layers $L_1$ to $L_a$.
\end{lemma}

\begin{proof}
Let $u$ be a point in $L_1 \cup \dots \cup L_{a}$.
Without loss of generality, suppose that $V$ is a right chain, and that the line $l(uz)$ is vertical with $u$ above $z$.
Consider the ray $h$ contained in $l(uz)$ that starts at $z$ and does not contain $u$.
Since $V$ wraps around, it crosses $h$ at least twice.
Therefore there are two non-consecutive edges $e_j$ and $e_k$ of $V$ that intersect $h$ (with $j<k$), and there is an edge $e_p$ between $e_j$ and $e_k$ that intersects the line segment $zu$.


Note that $u$ lies in $e_j^-$ and $e_k^-$, but $u$ lies in $e_p^+$.
Let $\tilde{V}$ be the maximal subchain of $V$ that contains $e_p$ and such that $u \in e^+$ for every edge $e$ of $\tilde{V}$.
Let $e_m$ and $e_n$ be the first and last edges of $\tilde{V}$.
Since $e_j$ and $e_k$ are not in $\tilde{V}$ and $j< m \leq n <k$, the edges $e_{m-1}$ and $e_{n+1}$ are 
not in $\tilde{V}$.
Thus $u \in e_{m-1}^\ominus \cap e_m^+$, as shown in Figure~\ref{fig:endsa}. 
Also, $v_m$ lies to the left of $l(uz)$ since $j<m\leq p$.
This implies that $u$ and $v_{m+1}$ are on the same side of $l(c_m)$, so $u \in c_m^\oplus$.
Furthermore, $u \in e_n^+ \cap e_{n+1}^\ominus$, and $v_{n+1}$ lies to the right of $l(uz)$ since $p\leq n <k$, as shown in Figure~\ref{fig:endsa} also.
This implies that $u \in d_n^\oplus$.


\begin{figure}
\begin{subfigure}[t]{0.5\textwidth}
\centering
\includegraphics{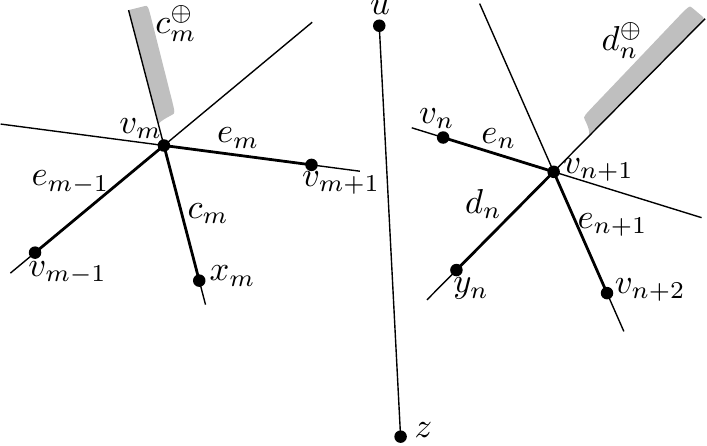} 
\caption{}
\label{fig:endsa}
\end{subfigure}
\qquad 
\begin{subfigure}[t]{0.4\textwidth}
\centering
\includegraphics{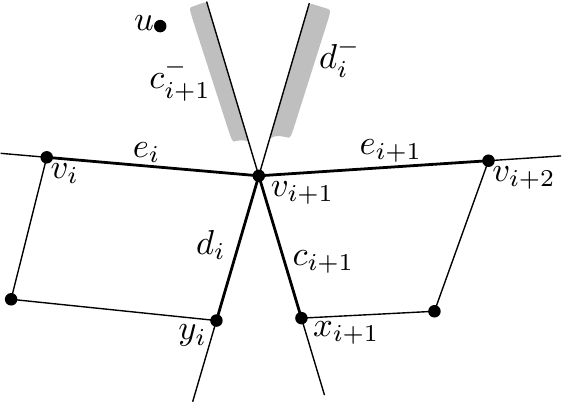}
\caption{}
\label{fig:endsb}
\end{subfigure}
\caption{(a) $u \in c_m^\oplus$ and $u \in d_n^\oplus$. (b) $u$ cannot be in both $d_i^-$ and $c_{i+1}^-$.}
\end{figure}


Since $u \in e_i^+ \cap e_{i+1}^+$ for $m \leq i \leq n-1$, 
the fact that $y_i$ precedes $x_{i+1}$ in $L_{i+2}$ (or $y_i = x_{i+1}$) means that it is not possible for $u$ to be in both $d_i^-$ and $c_{i+1}^-$; see Figure~\ref{fig:endsb}.
In order to prove that $u$ is in some $S[e_i]= c_i^\oplus \cap d_i^\oplus \cap e_i^\oplus$, it suffices to show that $u \in c_i^\oplus \cap d_i^\oplus$ for some $i \in \{m, \dots, n \}$.
Let $q$ be minimal such that $u \in d_q^\oplus$.
Such a $q$ exists because $u \in d_n^\oplus$.
Then either $q=m$ or $u \in d_{q-1}^-$, so in any case $u \in c_q^\oplus$.
Therefore $u$ lies in $S[e_q]$.
\end{proof}



Lemma \ref{wrap} says that if the number of middle layers $r-\ell-a$ is at least $32\ell$, then there is a chain $V=(v_1, \dots, v_t)$ with $t=32\ell$ that wraps around.
Since $P$ contains no empty pentagons, Lemma~\ref{cover} then implies that every point in the outer layers lies on one of the lines $l(c_i)$ or $l(d_i)$ that bound the sectors $S[e_i]$ corresponding to $V$. 
Thus the number of points in the outer layers is at most $2t(\ell-3)= 64\ell(\ell-3)$.
%
Recall however that $a$ was chosen so that the outer layers contained at least $64\ell(\ell-1)$ points,
 so in fact the number of middle layers is less than $32\ell$.
Therefore (by Theorem~\ref{convposthm}) the number of points in the middle layers is $\left| L_{a+1} \cup \dots \cup L_{r-\ell} \right| < 32\ell \times 8\ell = 256\ell^2$.
As noted at the beginning of the proof, 
$\left| L_1 \cup \dots \cup L_a \right| < 64\ell^2$,
and also $\left| L_{r-\ell+1} \cup \dots \cup L_r \right| < 8\ell^2$.
Adding everything up gives $|P| = \left| L_1 \cup \dots \cup L_r \right| < 328\ell^2$.
This contradicts the assumption that $|P| \geq 328\ell^2$, and so in fact $P$ does contain an empty pentagon.
This completes the proof of Theorem~\ref{mainthm}. \qed

\appendix
\section{Proof of Lemma~\ref{empty}}\label{append}

Lemma~\ref{empty} appears implicitly in the paper of Abel et al.~\cite{abeletal}. The following proof is adapted directly from that paper, and the figures are reproduced with the kind permission of the authors. For simplicity, consider a point set $P$ with $\ell$ layers, so the statement becomes: 

\theoremstyle{plain}
\newtheorem*{rlemma}{Lemma~\ref{empty}.1}
\begin{rlemma}
Let $L_1, \dots, L_\ell$ be the convex layers of a point set $P$. If $L_1$ contains an empty edge then $P$ contains an empty pentagon or $\ell$ collinear points.
\end{rlemma} 

\begin{proof}
Suppose for contradiction that $P$ contains no empty pentagon and no $\ell$ collinear points.
Let $z$ be a point in the innermost layer $L_\ell$ of $P$. 
Suppose $xy$ is an empty edge of $L_i$ for some $i \in \{1, \dots, \ell -2\}$.
In this case, the intersection of the boundary
of $\conv(L_{i+1})$ and $\Delta(x,y,z)$ is contained in an edge
$pq$ of $L_{i+1}$.  
Call $pq$ the \emph{follower} of $xy$.
First some properties of followers are established.


\begin{claim}\label{emptyfoll} If $pq$ is the follower of $xy$, then $pxyq$ is an empty quadrilateral and $pq$ is empty.
\end{claim}

\begin{proof}
  Let $Q:=pxyq$. Since $p$ and $q$ are in the interior of
  $\conv(L_i)$, both $x$ and $y$ are corners of $Q$.  Both $p$ and $q$
  are corners of $Q$, otherwise $xy$ would not be empty.
  Thus $Q$ is in strictly convex position. $Q$ is empty by the
  definition of $L_{i+1}$.

  Suppose that $pq$ is not empty; that is,
  $\Delta(p,q,z)\cap L_{i+2}\neq\emptyset$.  Then the $4$-sector $S(p,x,y,q) \neq \emptyset$, so $P$ contains an empty pentagon.  This contradiction proves that
  $pq$ is empty.
\end{proof}

As illustrated in \figref{Aligned}(a)--(c), the follower
$pq$ of $xy$ is said to be:
\begin{itemize}
\item \emph{double-aligned} if $p\in l(xz)$ and
  $q\in l(yz)$,
\item \emph{left-aligned} if $p\in l(xz)$ and
  $q\not\in l(yz)$,
\item \emph{right-aligned} if $p\not\in l(xz)$ and
  $q\in l(yz)$.
\end{itemize}

\Figure{Aligned}{\includegraphics{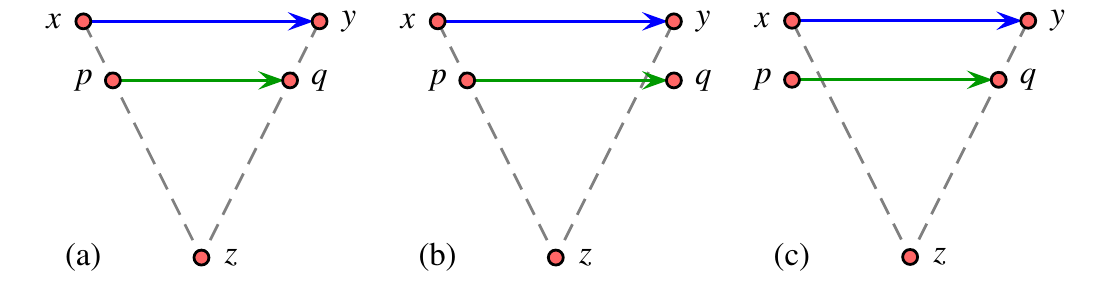}}{(a) Double-aligned. (b)
  Left-aligned. (c) Right-aligned.}

\begin{claim}\label{alignment} If ${pq}$ is the follower of ${xy}$, then ${pq}$ is either
  double-aligned or left-aligned or right-aligned.
\end{claim}

\begin{proof}
  Suppose that ${pq}$ is neither double-aligned nor
  left-aligned nor right-aligned, as illustrated in \figref{Hole}(a). By Claim~\ref{emptyfoll}, $pxyq$ is an empty quadrilateral. But the $4$-sector $S(p,x,y,q)$ contains the point $z$, so $P$ contains an empty pentagon.
\end{proof}

Returning to the proof of Lemma~\ref{empty}.1, let ${x_1y_1}$ be the empty edge of $L_1$.
For $i=2,3,\dots,\ell-1$, let ${x_iy_i}$ be the
follower of ${x_{i-1}y_{i-1}}$.  By Claim~\ref{emptyfoll} (at
each iteration), ${x_iy_i}$ is empty.  For some
$i\in\{2,\dots,\ell-2\}$, the edge ${x_iy_i}$ is not
double-aligned, as otherwise $\{x_1,x_2,\dots,x_{\ell-2},z\}$ are
collinear and $\{y_1,y_2,\dots,y_{\ell-2},z\}$ are collinear, which
implies that $\{x_1,x_2,\dots,x_{\ell-1},z\}$ are collinear or
$\{y_1,y_2,\dots,y_{\ell-1},z\}$ are collinear by Claim~\ref{alignment}.  Let
$i$ be the minimum integer in $\{2,\dots,\ell-2\}$ such that
${x_iy_i}$ is not double-aligned.  Without loss of
generality, ${x_iy_i}$ is left-aligned.  On the other
hand, ${x_jy_j}$ cannot be left-aligned for all
$j\in\{i+1,\dots,\ell-1\}$, as otherwise $\{x_1,x_2,\dots,x_{\ell-1},z\}$ are
collinear.  Let $j$ be the minimum integer in $\{i+1,\dots,\ell-1\}$ such that
${x_jy_j}$ is not left-aligned.  Thus
${x_{j-1}y_{j-1}}$ is left-aligned and
${x_jy_j}$ is not left-aligned.  It follows that
$x_{j-2}y_{j-2}y_{j-1}y_{j}x_{j-1}$ is an empty pentagon, as
illustrated in \figref{Hole}(b).  This contradiction completes the proof.
\end{proof}

\Figure{Hole}{\includegraphics{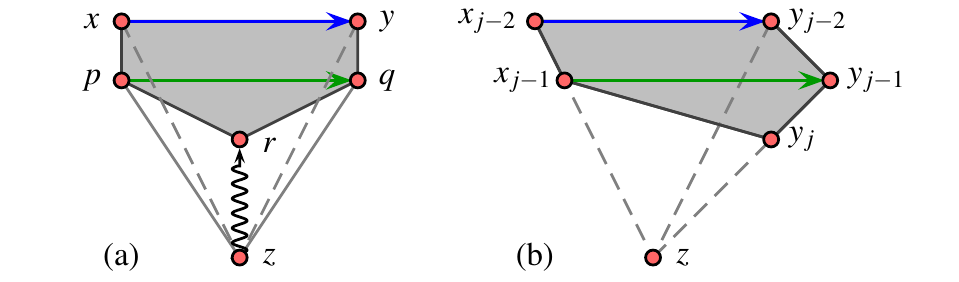}}{(a) Neither double-aligned nor
  left-aligned nor right-aligned. (b) The empty pentagon
  $x_{j-2}y_{j-2}y_{j-1}y_{j}x_{j-1}$.}

\bibliographystyle{siam}
\bibliography{references}


\end{document}